\documentclass[11pt]{article}
\usepackage{amsmath, amsthm, amscd, amssymb,tikz,cite,graphicx}
\usepackage{caption,subcaption}
\usepackage{thm-restate,hyperref,relsize,enumerate}

\numberwithin{equation}{section}

\theoremstyle{plain}
\newtheorem{theorem}{Theorem}[section]
\newtheorem*{thm-konig}{K\H onig--Egerv\'ary Theorem}

\newtheorem{observation}[theorem]{Observation}
\newtheorem{lemma}[theorem]{Lemma}

\newtheorem{claim}[theorem]{Claim}
\newtheorem{conjecture}{Conjecture}
\theoremstyle{definition}

\newtheorem{remark}{Remark}

\theoremstyle{plain}

\DeclareMathOperator{\SC}{\omega_S}

\begin{document}

\title{The strong clique number of graphs with forbidden cycles}

\author{\small Eun-Kyung Cho\thanks{
Department of Mathematics, Hankuk University of Foreign Studies, Yongin-si, Gyeonggi-do, Republic of Korea. \texttt{ekcho2020@gmail.com}
},  \  \ 
\small Ilkyoo Choi\thanks{
Department of Mathematics, Hankuk University of Foreign Studies, Yongin-si, Gyeonggi-do, Republic of Korea.
\texttt{ilkyoo@hufs.ac.kr}
},  \  \ 
\small Ringi Kim\thanks{
Department of Mathematical Sciences, Korea Advanced Institute of Science and Technology, Daejeon, Republic of Korea.
\texttt{kimrg@kaist.ac.kr}},  \  \  
 \small Boram Park\thanks{
Department of Mathematics, Ajou University, Suwon-si, Gyeonggi-do, Republic of Korea.
\texttt{borampark@ajou.ac.kr}}
,  \  \ 
}
\date\today
\maketitle
\begin{abstract}
Given a graph $G$, the {\it strong clique number} of $G$, denoted $\SC(G)$, is the maximum size of a set $S$ of edges such that every pair of edges in $S$ has distance at most $2$ in the line graph of $G$.
 As a relaxation of the renowned Erd\H{o}s--Ne\v set\v ril conjecture regarding the strong chromatic index, Faudree et al. suggested investigating the strong clique number, and 
 conjectured a quadratic upper bound in terms of the maximum degree. 

Recently, Cames van Batenburg, Kang, and Pirot conjectured a linear upper bound in terms of the maximum degree for graphs without even cycles. 
Namely, if $G$ is a $C_{2k}$-free graph, then $\SC(G)\leq (2k-1)\Delta(G)-{2k-1\choose 2}$, and if $G$ is a $C_{2k}$-free bipartite graph, then $\SC(G)\leq k\Delta(G)-(k-1)$. 
We prove the second conjecture in a stronger form, by showing that forbidding all odd cycles is not necessary. 
To be precise, we show that a $\{C_5, C_{2k}\}$-free graph $G$ with $\Delta(G)\ge 1$ satisfies $\SC(G)\leq k\Delta(G)-(k-1)$, when either $k\geq 4$ or $k\in \{2,3\}$ and $G$ is also $C_3$-free. 
Regarding the first conjecture, we prove an upper bound that is off by the constant term.
Namely, for $k\geq 3$, we prove that a $C_{2k}$-free graph $G$ with $\Delta(G)\ge 1$ satisfies $\SC(G)\leq (2k-1)\Delta(G)+(2k-1)^2$. 
This improves some results of Cames van Batenburg, Kang, and Pirot.

\end{abstract}

\maketitle

\section{Introduction}

Given a graph $G$, let $V(G)$, $E(G)$, and $\Delta(G)$ denote its vertex set, edge set, and maximum degree, respectively. 
We use $C_n$ and $P_n$ for a cycle and a path, respectively, on $n$ vertices.  
A graph $G$ is {\it $F$-free} for some graph $F$ if $G$ does not contain a subgraph isomorphic to $F$. 
The {\it distance} between two edges $e_1$ and $e_2$ of $G$ is the distance between  the two vertices corresponding to $e_1$ and $e_2$ in the line graph of $G$. 
In these terms, a {\it matching} is a collection of edges with pairwise distance at least 2, and an {\it induced matching} is a collection of edges with pairwise distance at least 3. 

Given a graph $G$, the {\it strong chromatic index} of $G$, denoted $\chi'_S(G)$, is the minimum $k$ such that $E(G)$ can be partitioned into $k$ induced matchings.
Since every edge has at most $2\Delta(G)^2-2\Delta(G)$ edges within distance 2, a greedy algorithm guarantees the following trivial bound: $\chi'_S(G)\leq 2\Delta(G)^2-2\Delta(G)+1$.
In 1985, Erd\H os and Ne\v set\v ril (see~\cite{bang2006six,faudree1990strong,halasz2012irregularities}) made the following renowned conjecture:

\begin{conjecture}[See \cite{bang2006six,faudree1990strong,halasz2012irregularities}]\label{conj:ErNe}
For a graph $G$,
\[\chi'_S(G)\leq
\begin{cases}
1.25 \Delta(G)^2 &\mbox{if $\Delta(G)$ is even }\\
1.25 \Delta(G)^2 -0.5 \Delta(G) +0.25 & \mbox{if $\Delta(G)$ is odd.}
\end{cases}
\]
\end{conjecture}

If Conjecture~\ref{conj:ErNe} is true, then it is sharp, as illustrated by blowing up each vertex of $C_5$ into an independent set of appropriate size. 
Despite the steady interest of numerous researchers, the conjecture still seems to be far from reach. 
We highlight two approaches with notable progress regarding Conjecture~\ref{conj:ErNe}.

One line of research focuses on reducing the coefficient of the leading term for graphs with sufficiently large maximum degree. 
Molloy and Reed~\cite{molloy1997bound} proved an upper bound of $1.9993\Delta(G)^2$, which was reduced to $1.9653\Delta(G)^2$ by Bruhn and Joos~\cite{bruhn2018stronger}.
Recently, a significant improvement was made by Bonamy, Perrett, and Postle~\cite{bonamy2018colouring}, who showed $1.835\Delta(G)^2$.

Another line of research tackles the conjecture for small maximum degrees. 
The only non-trivial case that is confirmed is  $\Delta(G)=3$, which was resolved by Andersen~\cite{andersen1992strong} and independently by Hor\'ak, Qing, and Trotter~\cite{horak1993induced}. 
The conjecture is open even for $\Delta(G)=4$.
Improving a result by Cranston~\cite{cranston2006strong}, Huang, Santana, and Yu~\cite{huang2018strong} recently proved that 21 colors suffice, whereas the conjectured bound is 20. 

As a variation of Conjecture~\ref{conj:ErNe}, researchers also considered classes of graphs with forbidden subgraphs. 
Unfortunately, the situation is not much better even for bipartite graphs. 
The following conjecture by Faudree et al.~\cite{faudree1990strong} is still open:

\begin{conjecture}[\cite{faudree1990strong}]\label{conj:bipartite}
For a  bipartite graph $G$, $\chi'_S(G)\leq \Delta(G)^2$. 
\end{conjecture}

If the above conjecture is true, then it is tight as demonstrated by the complete bipartite graphs with appropriate part sizes. 
As supporting evidence, the authors of \cite{faudree1990strong} proved that Conjecture \ref{conj:bipartite} is true for graphs where all cycle lengths are divisible by 4.
Note that all cycle lengths in bipartite graphs are divisible by 2. 
Steger and Yu~\cite{STEGER1993291} verified Conjecture~\ref{conj:bipartite} for $\Delta(G)=3$, which is the only known non-trivial maximum degree case.

Mahdian~\cite{mahdian2000strong} strengthened Conjecture~\ref{conj:bipartite} by asserting the same conclusion holds when only a $5$-cycle is forbidden, opposed to forbidding all odd cycles.

\begin{conjecture}[\cite{mahdian2000strong}]\label{conj:c5}
For a $C_5$-free graph $G$,  $\chi'_S(G)\leq \Delta(G)^2$. 
\end{conjecture}

Generalizing a result of Mahdian~\cite{mahdian2000strong}, who investigated $C_4$-free graphs, 
Vu~\cite{vu2002general} proved that the growth rate of the upper bound can actually be reduced by a logarithmic factor for $F$-free graphs, where $F$ is an arbitrary bipartite graph. 
Namely, for a bipartite graph $F$, there exists a constant $C_F$ such that if $G$ is an $F$-free graph with sufficiently large $\Delta(G)$, then $\chi'_S(G)\leq C_F\frac{\Delta(G)^2}{\log\Delta(G)}$.
Moreover, this result is tight up to a multiplicative constant factor.

\bigskip

A natural lower bound on a coloring parameter is the corresponding clique number. 
Given a graph $G$, a {\it strong clique} of $G$ is a set $S$ of edges such that every pair of edges in $S$ has distance at most $2$ in $G$.
The {\it strong clique number} of $G$, denoted $\SC(G)$, is the size of a maximum strong clique of $G$. 
As a weakening of Conjecture~\ref{conj:ErNe}, Faudree et al.~\cite{faudree1990strong} made the following conjecture:

\begin{conjecture}[\cite{faudree1990strong}]\label{conj:SC}
For a graph $G$, 
$$\SC(G)\leq
\begin{cases}
1.25 \Delta(G)^2 &\mbox{if $\Delta(G)$ is even }\\
1.25 \Delta(G)^2 -0.5 \Delta(G) +0.25 & \mbox{if $\Delta(G)$ is odd.}
\end{cases}$$
\end{conjecture}

If Conjecture~\ref{conj:SC} is true, then it is sharp by the same graph demonstrating the tightness of Conjecture~\ref{conj:ErNe}. 
In contrast to the discouraging status quo for solving Conjecture~\ref{conj:ErNe}, there has been significant progress on Conjecture~\ref{conj:SC}.
The authors of~\cite{faudree1990strong} proved the existence of $\varepsilon>0$ such that $\SC(G)\leq (2-\varepsilon)\Delta(G)^2$ for  sufficiently large $\Delta(G)$. 
After successive improvements by Bruhn and Joos~\cite{bruhn2018stronger} and {\'S}leszy{\'n}ska-Nowak~\cite{sleszynska2016clique},  Faron and Postle~\cite{faron2019clique} recently proved that $\SC(G)\leq \frac{4}{3}\Delta(G)^2$.

We point out that Chung et al.~\cite{CHUNG1990129} proved that a graph where every pair of edges has distance at most 2 has at most $1.25\Delta(G)^2$ edges;  this is different from the strong clique number since a strong clique does not necessarily contain all edges of the host graph. 

We now redirect our attention to the strong clique number of graphs with forbidden subgraphs. 
As supporting evidence for Conjecture~\ref{conj:bipartite}, Faudree et al.~\cite{faudree1990strong} proved that a bipartite graph $G$ has strong clique number at most $\Delta(G)^2$. 
Note that this is tight as equality holds for the same graph demonstrating the tightness of Conjecture~\ref{conj:bipartite}. 
It was recently revealed that forbidding all odd cycles is not necessary, as Cames van Batenburg, Kang, and Pirot~\cite{van2020strong} showed that a $C_5$-free graph $G$ has strong clique number at most $\Delta(G)^2$. 
This enhancement verifies the strong clique version of  Conjecture~\ref{conj:c5}, as well as Conjecture~\ref{conj:SC} with a much better upper bound. 
Note that equality holds for complete bipartite graphs with appropriate part sizes.
 
Cames van Batenburg, Kang, and Pirot~\cite{van2020strong} also considered the class of graphs with other forbidden odd cycles. 
They proved that a $C_3$-free graph $G$ has strong clique number at most $1.25\Delta(G)^2$, which is tight for blowups of $C_5$. 
Note that this proves Conjecture~\ref{conj:SC} for $C_3$-free graphs.
They also proved that for $k\geq 3$, if $G$ is a $C_{2k+1}$-free graph where $\Delta(G)\geq 3k^2+10k$, then $\SC(G)\leq \Delta(G)^2$.
This implies the strong clique version of Conjecture~\ref{conj:bipartite}, as well as Conjecture~\ref{conj:SC} with a much better upper bound for graphs with a forbidden odd cycle and large maximum degree.

The authors of \cite{van2020strong} speculated that the situation is much different for the class of graphs with a forbidden even cycle. 
In contrast to the quadratic upper bounds of all aforementioned conjectures, they put forth the below conjecture asserting a linear upper bound:

\begin{conjecture}[\cite{van2020strong}]\label{conj:general}
For a $C_{2k}$-free graph $G$ with $\Delta(G)\ge 1$, $\SC(G) \leq (2k-1)\Delta(G)-{2k-1\choose 2}$. 
\end{conjecture}

If Conjecture~\ref{conj:general} is true, then it is sharp as exhibited by the following graph $H$: attach $\Delta(H)-(2k-2)$ pendent edges to each vertex of a complete graph on $2k-1$ vertices. 
When $k=2$, however, Conjecture~\ref{conj:general} is false, since a $5$-cycle is $C_4$-free and the entire graph is a strong clique with five edges. 
For other values of $k$, Conjecture~\ref{conj:general} seems plausible. 
As evidence, 
the authors of~\cite{van2020strong} proved the following:

\begin{theorem}[\cite{van2020strong}]
Let $G$ be a graph with $\Delta(G)\ge 1$. 
\label{thm:van1}
\begin{itemize}
\item[\rm(i)] If $G$ is $C_4$-free and $\Delta(G)\geq 4$, then $\SC(G)\leq 3\Delta(G)-3$.
\item[\rm(ii)] For $k\geq 3$, 
if $G$ is $C_{2k}$-free, then $\SC(G)\leq 10k^2\Delta(G)-10k^2$.
\item[\rm(iii)] For $k\ge 2$, if $G$ is $\{C_{2k}, C_{2k+1}, C_{2k+2}\}$-free, then $\SC(G)\leq (2k-1)\Delta(G)-(2k-3)$.
\end{itemize}
\end{theorem}

Note that (i) in the above theorem resolves Conjecture~\ref{conj:general} in the affirmative when $k=2$ and $\Delta(G)$ is not so small. 
The authors of~\cite{van2020strong} also put forth the following conjecture for bipartite graphs with a forbidden even cycle. 

\begin{conjecture}[\cite{van2020strong}]\label{conj:bipart}
For a $C_{2k}$-free bipartite graph $G$ with $\Delta(G)\ge 1$,  $\SC(G) \leq k\Delta(G)-(k-1)$. 
\end{conjecture}

If Conjecture~\ref{conj:bipart} is true, then it is sharp for the following graph: attach $p$ pendent edges to one vertex of degree $k-1$ in a complete bipartite graph $K_{k-1,p+k-1}$. As evidence for Conjecture~\ref{conj:bipart}, the following theorem was shown:

\begin{theorem}[\cite{van2020strong}]\label{thm:manycycles}
If $G$ is a $\{C_3, C_5, C_{2k}, C_{2k+2}\}$-free graph, then $\SC(G)\leq \max\{k\Delta(G), 2k(k-1)\}$. 
\end{theorem}

Our first contribution is that we verify Conjecture~\ref{conj:bipart} in a much stronger form. 
Theorem~\ref{thm:conj:biparite} resolves Conjecture~\ref{conj:bipart} in the affirmative. 

\begin{theorem}\label{thm:conj:biparite}
For $k\ge 2$, if $G$ is a $C_{2k}$-free bipartite graph and $\Delta(G)\ge 1$, then $\SC(G) \leq k\Delta(G)-(k-1)$. 
\end{theorem} 
 
We strengthen Theorem~\ref{thm:conj:biparite} and obtain Theorem~\ref{thm:main:coro}, which improves aforementioned results by Cames van Batenburg, Kang, and Pirot. 
Namely, we prove that the same conclusion can be reached by forbidding only $C_5$ and $\{C_5, C_3\}$ when $k\geq 4$ and $k=3$, respectively, opposed to forbidding all odd cycles.

\begin{theorem}\label{thm:main:coro}
For $k\ge 2$,
if $G$ is a $\{C_5,C_{2k}\}$-free graph and $\Delta(G)\ge 1$, 
then the following holds: 
\begin{itemize}
    \item[\rm(i)] For $k \ge 4$,  $\SC(G)\leq k\Delta(G)-(k-1)$. 
    \item[\rm(ii)] For $k\in\{2,3\}$, if $G$ is also $C_3$-free, then
    $\SC(G)\leq k\Delta(G)-(k-1)$. 
\end{itemize}
\end{theorem}

When $k\in \{2,3\}$, forbidding $C_3$ in (ii) is necessary as demonstrated by the following graph $H$: 
attach $p$ pendent edges 
to each vertex of a complete graph on $k+1$ vertices.
This graph is $\{C_5,C_{2k}\}$-free, yet $\SC(H)=(k+1)\Delta(H)-k(k+1)/2 >k\Delta(H)-(k-1)$ when $\Delta(H)$ is sufficiently large.

Our second contribution is that we almost prove Conjecture~\ref{conj:general}.
We are able to provide an upper bound that is off by only the constant term. 
Note that Theorem~\ref{thm:general} is a strengthening of Theorem~\ref{thm:van1} (ii) and (iii). 
\begin{theorem}\label{thm:general}
For $k\geq 3$, if $G$ is a $C_{2k}$-free graph and $\Delta(G)\ge 1$, then $\SC(G)\le (2k-1)\Delta(G)+(2k-1)^2$. 
\end{theorem}

The paper is organized as follows. 
We first provide some definitions and prove some lemmas in Section~\ref{sec:prelmi}.
In Section~\ref{sec:bipartite}, we prove Theorem~\ref{thm:conj:biparite}, which is used to show Theorem~\ref{thm:main:coro}.
The proof of Theorem~\ref{thm:general} is provided in Section~\ref{sec:general}.

\section{Preliminaries}\label{sec:prelmi}

We provide some  definitions and useful observations in this section. 

Given a graph $G$, let $S$ (resp. $W$) be a subset of the edges (resp. vertices) of $G$. 
We use $G[S]$ (resp. $G[W]$) to denote the subgraph of $G$ induced by the edges in $S$ (resp. the vertices in $W$).
Let $G-S$ (resp. $G-W$) denote the graph obtained from $G$ by deleting the edges in $S$ (resp. vertices in $W$).
If $S = \{uv\}$ (resp. $W= \{v\}$), then denote $G-S$ by $G-uv$ (resp. $G-W$ by $G-v$). 

Given a graph $G$ and $A, B\subseteq V(G)$, 
let $E_G(A,B)$ denote the set of all edges in $G$ 
joining a vertex in $A$ and a vertex in $B$.
When we denote a cycle or a path, we drop commas for simplicity. For instance, a (directed) path $x_1,x_2,\ldots,x_n$ of length $n-1$ and a (directed) cycle $x_1,x_2,\ldots,x_n,x_1$ is denoted  by  $x_1x_2\ldots x_n$ and $x_1x_2\ldots x_nx_1$, respectively.

A {\it vertex cover} of a graph $G$ is a set $S\subseteq V(G)$ such that every edge of $G$ has an endpoint in $S$. The {\it vertex cover number} of $G$, denoted $\tau(G)$, 
is the size of a minimum vertex cover of $G$. 
The {\it matching number} of $G$, denoted $\alpha'(G)$,  is the size of a maximum matching in $G$. 
For a set of edges $M$, let $V(M)$ denote the set of all endpoints of edges in $M$.
The following is arguably the most famous theorem relating the matching number and the vertex cover number of bipartite graphs. 

\begin{thm-konig}[\cite{konig1931grafok, egervary1931matrixok}]
\label{thm:konig}
If $G$ is a bipartite graph, then $\tau(G)=\alpha'(G)$.
Moreover, for a maximum matching $M$ of $G$, there is a minimum vertex cover that is a subset of $V(M)$.
\end{thm-konig}

We now prove two lemmas that will be often used in the proofs of our theorems. 

\begin{lemma}\label{lem:matching:path}
Let $G$ be a bipartite graph, and $H$ be the subgraph of $G$ induced by a strong clique of $G$.
If $H$ has a matching $M$ of size $m$, 
then $G[V(M)]$ contains a $P_{2m}$ 
containing all  edges in $M$. Moreover, if $m\ge 4$, then $G[V(M)]$ contains 
a $C_{2m-2}$ 
using at least $m-2$ edges in $M$.
\end{lemma}
\begin{proof}
Let $G$ be a bipartite graph with bipartition $(X,Y)$,
and let $M=\{x_1y_1,\ldots,x_{m}y_{m}\}$ be a matching of $H$ where $\{x_1, \ldots, x_m\}\subseteq X$ and $\{y_1, \ldots, y_m\}\subseteq Y$. 
Construct an auxiliary directed graph $D$ where each vertex $w_i$ of $D$ represents the edge $x_iy_i$ in $M$ 
and $(w_i, w_j)$ is an arc of $D$ if $x_iy_j$ is an edge of $G$.  
Note that $D$ is a semi-complete digraph since for distinct $i$ and $j$, 
either $x_iy_j$ or $x_jy_i$ exists in $G$.
Since $D$ contains a tournament, which always has a Hamiltonian path, 
we may assume that $w_1 w_2\ldots w_m$ is a directed path in $D$ by relabelling indices if necessary.
Thus $G[V(M)]$ has a path $y_1x_1 y_2 x_2\ldots y_m x_m$ of length $2m-1$ containing all edges in $M$.

Suppose that $m\ge 4$.  
If $D$ is strongly connected, then $D$ contains a directed cycle of every length. 
Thus, $D$ has a directed cycle of length $m-1$, which corresponds to a 
$C_{2m-2}$ 
in $G[V(M)]$ using $m-1$ edges in  $M$.
If $D$ is not strongly connected, then by the acyclic ordering of strongly connected components, there is a directed $(u,v)$-path of length $m-1$ for some $u,v$ where $(u,v)$ is an arc of $D$.
This corresponds to a $C_{2m-2}$
 in $G[V(M)]$ using $m-2$ edges  in $M$.
\end{proof}

For a strong clique $S$ of a graph $G$, 
the graph $G$ is \textit{$S$-minimal} if 
$S$ is not a strong clique of every proper subgraph of $G$. In other words, removing any vertex or edge of $G$ would violate that $S$ is a strong clique. 

\begin{lemma}
\label{lem:minimal}
Let $S$ be a strong clique of a graph $G$.
If $G$ is $S$-minimal, then the following holds:
\begin{enumerate}[\rm(i)]
    \item Every vertex of $G$ is incident with some edge in $S$.
    \item The diameter of $G$ is at most $3$. 
    \item For every edge $uv \in E(G)\setminus S$, there are two edges $uu', vv'\in S$ such that $uv$ is the only edge joining $uu'$ and $vv'$. 
    \item If $S$ is a maximum strong clique of $G$, then for every edge $uv \in E(G)\setminus S$, 
there is an edge $xy \in S$ whose distance from $uv$ in $G$ is at least $3$. 
\end{enumerate}
\end{lemma}
\begin{proof}
(i) Suppose that $G$ has a vertex $v$ that is not incident with an edge in $S$. 
Since the distance between two edges in $S$ is the same in both $G$ and $G-v$, $S$ is also a strong clique of $G-v$.
This is a contradiction to the assumption that $G$ is $S$-minimal.

(ii) Suppose that $G$ has two vertices $u$ and $v$ where the distance between $u$ and $v$ is at least $4$. 
Then, for each $uu', vv' \in S$, the distance between $uu'$ and $vv'$ is at least $3$.
This is a contradiction to the assumption that $S$ is a strong clique of $G$.

(iii) Let $uv \in E(G) \setminus S$.
Suppose that $G$ has an edge other than $uv$ that joins $uu'$ and $vv'$ for every two edges $uu', vv' \in S$.
Then, for each pair of edges in $S$, the distance between them is the same in both $G-uv$ and $G$. 
Thus $S$ is also a strong clique of $G-uv$.
This is a contradiction to the assumption that $G$ is $S$-minimal.

(iv) Suppose that $S$ is a maximum strong clique of $G$, and $uv \in E(G) \setminus S$. 
If every edge in $S$ has distance at most $2$ from $uv$, then $S \cup \{uv\}$ is also a strong clique of $G$.
This is a contradiction to the assumption that $S$ is a maximum strong clique of $G$. 
\end{proof}

We end this section with a result from~\cite{faudree1990strong}.

\begin{theorem}[\cite{faudree1990strong}]\label{thm:helppp}
For a bipartite graph $G$, $\SC(G) \leq \Delta(G)^2$. 
\end{theorem}

\section{Proofs of Theorems~\ref{thm:conj:biparite} and \ref{thm:main:coro}}\label{sec:bipartite}

In this section, we first prove Theorem~\ref{thm:conj:biparite}, then show Theorem~\ref{thm:main:coro} by using Theorem~\ref{thm:conj:biparite}. 

\begin{proof}[Proof of Theorem~\ref{thm:conj:biparite}]

Let $k\ge 2$, and $G$ be a $C_{2k}$-free bipartite graph with bipartition $(X,Y)$ and $\Delta(G)\ge 1$.
Let $H$ be the subgraph of $G$ induced by a maximum strong clique of $G$. 
Recall that our goal is to show $|E(H)|\le k \Delta(G) -(k-1)$.

We first consider the case when $k=2$.
Let $M$ be a maximum matching of $H$ where 
$M=\{x_1y_1,\ldots,x_my_m\}$ and $x_i \in X$ and $y_i \in Y$ for every $i \in \{1,2,\ldots,m\}$.
The theorem is trivial when $m=1$, so assume $m\ge 2$. 
Without loss of generality assume $x_1y_2 \in E(G)$.
If either $\deg_H(y_1) \ge 2$ or $\deg_H(x_2) \ge 2$, then $G$ contains a $C_4$.  
Thus, if $m=2$, then 
$|E(H)|\le 2\Delta(G)-1$.
Now, we assume $m\ge 3$. 
If $x_1$ has three neighbors $y_1,y_2,y_j$ in $V(M)$ for some $3\le j \le m$, 
then since there must be an edge between $\{x_2,y_2\}$ and $\{x_j,y_j\}$, 
either  $x_1y_{2}x_{2}y_{j}x_1$  or $x_1y_{2}x_{j}y_{j}x_1$ is a $C_4$ in $G$, which is a contradiction. 
Hence, $x_1$ has at most two neighbors $y_1$ and $y_2$  in $V(M)$.
Similarly, every vertex in $V(M)$ has at most two neighbors in $V(M)$.
This implies that $m=3$, and $G[V(M)]$ is a $6$-cycle $x_1y_2x_2y_3x_3y_1x_1$.
 If $E(H)\setminus M$ is non-empty, then without loss of generality, let $x_1z \in E(H)\setminus M$. 
 Since $x_1z, x_3y_3 \in E(H)$, it follows that $x_3z \in E(G)$ and $x_1zx_3y_1x_1$ is a $C_4$, which is a contradiction.
Therefore, $E(H)=M$, so $|E(H)|=3 \le 2\Delta(G)-1$.

Now, we suppose that $k\ge 3$.
We may assume  $G$ is $E(H)$-minimal by removing unnecessary vertices and edges of $G$.
If $\Delta(G)\leq k-1$, then $|E(H)|\le \Delta(G)^2 \le k \Delta(G)-(k-1)$ where the first inequality holds by Theorem~\ref{thm:helppp}, so we may assume that $\Delta(G) \ge k$.

By Lemma~\ref{lem:matching:path}, 
we may assume that $H$ does not contain a matching of size $k+1$, so
by the K\H onig--Egerv\'ary Theorem, $\tau(H)\le k$.
If either $\Delta(H) < \Delta(G)$ or $\tau(H) <k$, then $|E(H)|\le \tau(H)\Delta(H)< k \Delta(G)-(k-1)$.
So, we may assume that $\Delta(H)=\Delta(G)$ and $\tau(H)=k$.

Let $Z$ be a minimum vertex cover of $H$.
If each vertex in $Z$ has degree less than $\Delta(H)$ in $H$, 
then $|E(H)|\le |Z|(\Delta(H)-1) < k \Delta(G)-(k-1)$. Hence, we may assume that there exists $z \in Z$ such that $\deg_H(z)=\Delta(H)=\Delta(G)$.
Without loss of generality, assume that $z \in X$.

Suppose that $zy \notin E(G)$ for some $y \in Y$. 
Since $G$ is $E(H)$-minimal, $y$ is incident with an edge $xy$ of $H$ by Lemma~\ref{lem:minimal} (i). 
Now, $x$ is adjacent to $z'$ for every $z' \in N_G(z)$ since $zz', xy \in E(H)$. 
See the first figure of Figure~\ref{fig:thm1.3}.
\begin{figure}[h!]
    \centering
    \includegraphics[height=4cm,page=2]{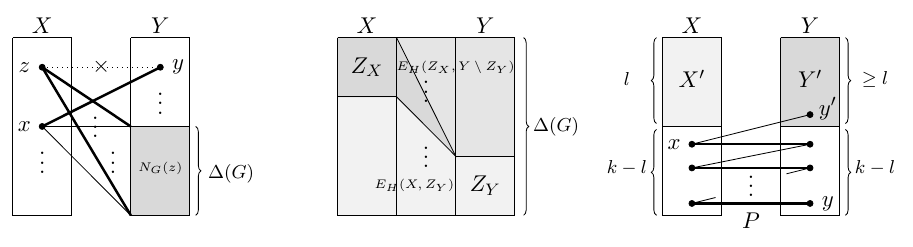}
    \caption{Illustrations for Theorem~\ref{thm:conj:biparite}. Thick edges are guaranteed to be in $H$.}
    \label{fig:thm1.3}
\end{figure}
Therefore, $\deg_G(x)\ge \deg_G(z)+1 =\Delta(G)+1$, which is a contradiction.
Hence, $Y=N_G(z)=N_H(z)$.

Suppose that $Z\cap Y\neq \emptyset$.
Let $Z_X=Z\cap X$ and $Z_Y=Z\cap Y$ so that $|Z|=|Z_X|+|Z_Y|=k$. 
Since $Z$ is a vertex cover of $H$, each edge of $H$ is in either 
$E_H(X, Z_Y)$ or $E_H(Z_X,Y\setminus Z_Y)$. 
See the second figure of Figure~\ref{fig:thm1.3}. 
Since each vertex in $Z_Y$ has degree at most $\Delta(G)$ in $H$, we know $|E_H(X, Z_Y)|\le |Z_Y|\Delta(G)$.
Also, since $|Y|=\Delta(G)$, we know $|E_H(Z_X,Y\setminus Z_Y)| \le |Z_X| (\Delta(G)-|Z_Y|)$.
Therefore,
\[
|E(H)|\le |Z_Y|\Delta(G)+|Z_X| (\Delta(G)-|Z_Y|) = (|Z_X|+|Z_Y|) \Delta(G)-|Z_X||Z_Y| \le k\Delta(G)-(k-1).
\]
The last inequality holds since both $Z_X$ and $Z_Y$ are not empty. 

Now, suppose that $Z\subseteq X$. 
If $x\in X\setminus Z$, then since $G$ is $E(H)$-minimal, 
$x$ is incident with an edge $xy'$ of $H$ by Lemma~\ref{lem:minimal} (i). 
This is a contradiction since $xy'$ is not covered by $Z$. 
Therefore, $X\setminus Z=\emptyset$, so $X=Z$ and thus $|X|=k$.

Let $X'=\{x\in X\mid N_G(x)=Y\}$, and let $|X'|=\ell$.
Since $z\in X'$ and $G$ is $C_{2k}$-free, we know $1 \le \ell \le k-1$.
By the K\H onig--Egerv\'ary Theorem, since $\tau(H)=k$, there is a matching of size $k$ in $H$.
Thus, $H-X'$ has a matching of size $k-\ell$, and this matching is a strong clique of $G$. 
Thus, it follows from  Lemma~\ref{lem:matching:path} that there is a path $P$ of length $2(k-\ell)-1$ in $G-X'$ using all vertices in $X\setminus X'$. 
Let $x$ and $y$ be the ends of $P$ in $X\setminus X'$ and $Y$, respectively.
Let $Y'=Y\setminus V(P)$. 
Note that  $|Y'|=\Delta(G)-(k-\ell) \ge \ell$.
If $x$ has a neighbor $y'$ in $Y'$, 
then we can extend $P$ to a $C_{2k}$ by using all vertices in $X'$ and $\ell$ vertices in $Y'$ including $y'$, which is a contradiction.
See the third figure of Figure~\ref{fig:thm1.3}.
Hence, $x$ has no neighbor in $Y'$. Then, 
\begin{align*}
|E(H)|&=|E_H(X',Y)|+|E_H(\{x\},Y)|+|E_H(X\setminus (X'\cup \{x\}),Y)|\\
&\le \ell\Delta(G) + (k-\ell)+(k-\ell-1)(\Delta(G)-1)\\
&\le (k-1)\Delta(G)+1\\
&\le k\Delta(G)-(k-1),
\end{align*}
where the last inequality holds since $\Delta(G)\ge k$.
This completes the proof.
\end{proof}

In order to prove Theorem~\ref{thm:main:coro}, we first show the following two lemmas. 

\begin{lemma} \label{lem:H-minimal:clique}
Let $G$ be a $C_5$-free graph, and $H$ be the subgraph of $G$ induced by a maximum strong clique of $G$.
If $H$ is $C_3$-free, and $G$ is $E(H)$-minimal, then $G$ is bipartite.
\end{lemma}

\begin{proof}
We will show that $G$ does not contain an odd cycle.
We first show that $G$ has no $C_3$.

Suppose to the contrary that $G$ has a $C_3$. Let $xyzx$ be a $C_3$ of $G$ incident with the maximum number of edges of $H$.
Since $H$ has no $C_3$, we may assume that $xy \not \in E(H)$.
Since $G$ is $E(H)$-minimal, 
there are edges $xx'$ and $yy'$ in $H$ whose distance is at least $3$ in $G-xy$ by Lemma~\ref{lem:minimal} (iii).
Moreover, since $xy\not\in E(H)$ and $E(H)$ is a maximum strong clique of $G$, there is an edge $uv\in E(H)$ 
whose distance from $xy$ is 
at least $3$ in $G$ by Lemma~\ref{lem:minimal} (iv).
Thus, $x,y,z,x',y',u,v$ are all distinct. 
From the pairwise distances between $xx',yy'$, and $uv$,
we may assume that $ux',vy'\in E(G)$ since $G$ is $C_5$-free. 
Also, $zu, zv \not\in E(G)$ since if $zu \in E(G)$ and $zv \in E(G)$, then $zuvy'yz$ and $zvux'xz$, respectively, is a $C_5$ of $G$. 
This further implies that $uv$ has distance at least $3$ to each of $xz$ and $yz$ in $G$.
So, $xz, yz \notin E(H)$. See the first figure of Figure~\ref{fig:H-minimal:clique}.

\begin{figure}[h!]
    \centering
    \includegraphics[height=2.5cm,page=4]{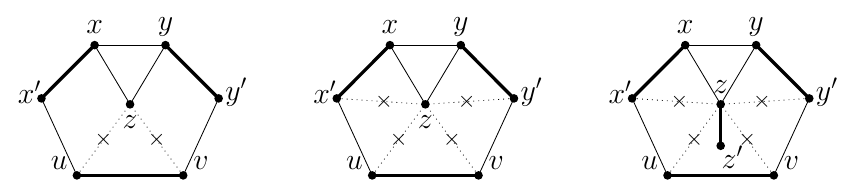}
    \caption{Illustrations for Lemma~\ref{lem:H-minimal:clique}}
    \label{fig:H-minimal:clique}
\end{figure}

Also, $zx', zy'\not\in E(G)$ since if $zx'\in E(G)$ and $zy'\in E(G)$, then $zx'x$ and $zy'y$, respectively, is a $C_3$ of $G$ containing an edge of $H$,
which is a contradiction to the choice of $xyzx$.
See the second figure of Figure~\ref{fig:H-minimal:clique}.
Since $G$ is $E(H)$-minimal, there is an edge $zz'\in E(H)$ where $z'\not\in \{x,y,z,u,v,x',y'\}$ by Lemma~\ref{lem:minimal} (i).
See the third figure of Figure~\ref{fig:H-minimal:clique}.
Since the distance between $zz'$ and $uv$ must be at most $2$, 
either $z'u\in E(G)$ or $z'v\in E(G)$.
In either case, $z'ux'xz$ or  $z'vy'yz$ is a $C_5$ of $G$, which is a contradiction.
Therefore, $G$ is $C_3$-free.

Now, we prove that $G$ is bipartite.
Suppose to the contrary that $G$ is not bipartite, so let  $C:x_1x_2\ldots x_{2m+1}x_1$ be a smallest odd cycle in $G$ 
with the maximum number of edges in $H$.
By the minimality of $|C|$, $C$ has no chords in $G$.
Moreover, $m=3$ since the diameter of $G$ is at most $3$ by Lemma~\ref{lem:minimal} (ii). 
Since $C$ has no chords and $G$ is $C_5$-free, for each $i\in\{1,2,\ldots,7\}$, the distance between $x_i$ and $x_{i+3}$ is exactly $3$ in $G$ where addition in the indices is modulo 7.
Since $C$ has no chords in $G$, there are two consecutive edges, say $x_1x_2, x_2x_3$, of $C$ not in $H$.
Since $G$ is $E(H)$-minimal, there are edges $x_1y_1, x_4y_4$ of $H$ by Lemma~\ref{lem:minimal} (i).
Since  
the distance between $x_1$ and $x_4$ is $3$,
it follows that $y_1y_4 \in E(G)$.
We also have $y_1,y_4 \notin\{x_1,x_4,x_5,x_6,x_7\}$ since 
the distances between $x_1$ and $x_5$ and between $x_4$ and $x_7$ are exactly $3$.
Note that the cycle $x_1y_1y_4x_4x_5x_6x_7x_1$ has more edges of $H$ than $C$, because $x_1y_1, x_4y_4 \in E(H)$ but $x_1x_2, x_2x_3 \notin E(H)$, which is a contradiction to the choice of $C$.
Therefore, $G$ is bipartite. 
\end{proof}

\begin{lemma} \label{lem:H:has:triangle}
Let $G$ be a $C_5$-free graph with $\Delta(G)\ge 1$, and $H$ be the subgraph of $G$ induced by a maximum strong clique of $G$. 
If $H$ contains a $C_3$, then $\SC(G)\le 4\Delta(G)-3$.
\end{lemma}

\begin{proof}
Let $C:xyzx$ be a $C_3$ of $H$.
Suppose to the contrary that $H-\{x,y,z\}$ has two edges $uv$ and $u'v'$, where $u,v,u',v'$ are all distinct. 
Since $E(H)$ is a strong clique of $G$ and $G$ is $C_5$-free, 
we may assume that $ux,uy\in E(G)$.
Similarly, 
we may assume that $u'$ is adjacent to two vertices of $C$.
Since $G$ is $C_5$-free, $u'$ is adjacent to both $x$ and $y$.
By the distance between $uv$ and $u'v'$, an edge of $G$ connects $\{u,v\}$ and $\{u',v'\}$. 
In each case, however, we can find $C_5$, which is a contradiction.

Thus $H-\{x,y,z\}$ is a star, and let $v$ be its  center vertex.
Then $\{v,x,y,z\}$ is a vertex cover of $H$, and so 
$|E(H)|\le 4\Delta(H)-3\le 4\Delta(G)-3$.
\end{proof}

\begin{proof}[Proof of Theorem~\ref{thm:main:coro}]
Let $G$ be a $\{C_5,C_{2k}\}$-free graph with $\Delta(G)\ge 1$.
Let $H$ be the subgraph of $G$ induced by a maximum strong clique of $G$. 
We may assume that $G$ is $E(H)$-minimal by removing unnecessary vertices and edges of $G$.
If $H$ does not contain a $C_3$, then by Lemma~\ref{lem:H-minimal:clique}, $G$ is bipartite, and so by Theorem~\ref{thm:conj:biparite}, it holds that $|E(H)|\le k\Delta(G)-(k-1)$.

If $H$ contains a $C_3$, then it is case (i) and so $k\ge  4$, and so $|E(H)|\le 4\Delta(G)-3\le k\Delta(G)-(k-1)$ by Lemma~\ref{lem:H:has:triangle}.
\end{proof}

\section{Proof of Theorem~\ref{thm:general}}\label{sec:general}

In this section, we prove Theorem~\ref{thm:general}.
Let $M$ be a matching of a graph $G$. For a vertex $x\in V(M)$,
an {\it $(x, M)$-path} is a path $P$ in $G[V(M)]$ starting with $x$ such that the last edge of $P$ is not in $M$ and for every distinct $u,v \in V(P)$, 
if $uv$ is an edge in $M$, then $uv \in E(P)$.

For a vertex $x \in V(M)$, let $x'$ denote the neighbor of $x$ such that $xx'\in M$. 
Note that $(x')'=x$.
A matching $M=\{x_1x_1',\ldots,x_mx_m'\}$ is {\it $x_1$-special} if 
$x_i x_j \in E(G)$ if and only if $i=1\neq j$, $x_i' x_j' \in E(G)$ if and only if $2\le i<j \le m$, and $x_ix_j' \notin E(G)$ for $1\le i\neq j \le m$.
In other words, an $x_1$-special matching of size $m$ can be obtained from a complete graph on  $m$ vertices by subdividing all edges incident with a vertex $x_1$ and adding a pendent edge to $x_1$. See Figure~\ref{fig:sp_matching} for an illustration. 
We say $M$ is {\it special} if it is $x$-special for some vertex $x\in V(M)$.
\begin{figure}[h!]
    \centering
    \includegraphics[height=4cm,page=8]{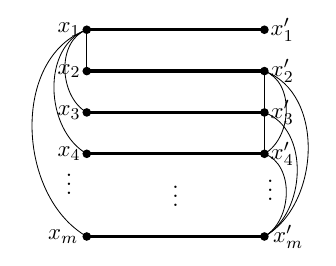}
    \caption{An illustration for an $x_1$-special matching $M$. Note that the vertices not adjacent in the figure are not adjacent in $G$.}  \label{fig:sp_matching}
\end{figure}

 We will use the following observation frequently. 

\begin{observation}\label{obs:special}
For a graph $G$, let $M$ be a matching of $G$ that is also a strong clique of $G$.
Suppose that $|M|=m\ge 3$. 
If $x\in V(M)$, then the following holds:
\begin{enumerate}[(a)]
\item There is no $(x, M)$-path of length $2$ if and only if $M$ is $x$-special.
\item If $M$ is $x$-special, then for every $\ell\in \{1,\ldots,m\}\setminus\{2\}$, there is an $(x, M)$-path of length $\ell$. 
\end{enumerate}

\end{observation}

\begin{lemma}\label{lem:path1}
For a graph $G$, let $M$ be a matching that is also a strong clique of $G$.
Suppose that $|M|=m\geq 2$ and $M$ is not $x$-special for some $x \in V(M)$.
If $x$ has a neighbor in $V(M)\setminus \{x,x'\}$,
then there is an $(x, M)$-path of length $\ell$ for every $\ell\in\{1,\ldots, m-1\}$.
\end{lemma}
\begin{proof}

We use induction on $\ell$ to prove Lemma~\ref{lem:path1}. Let $y$ be a neighbor of $x$ in $V(M)\setminus\{x, x'\}$. 

If $\ell=1$, then $xy$ is an $(x, M)$-path of length $1$.
If $\ell=2$, then $m\ge 3$, and by  Observation~\ref{obs:special}~(a),
there is an $(x, M)$-path of length $2$. 

Assume that $\ell \ge 3$, so $m\ge 4$.
Let $M'=M\setminus\{xx'\}$. Note that $|M'|\ge 3$.
If there is an $(y, M')$-path of length $\ell-1$, then by prepending $xy$ to the path, we obtain an $(x, M)$-path of length $\ell$.
So, let us assume that no such path exists.
By the induction hypothesis, either
$M'$ is $y$-special or $y$ has no neighbors in $V(M')\setminus\{y, y'\}$. 
By Observation~\ref{obs:special}, 
there are only two possible cases: 
either $M'$ is $y$-special and $\ell-1=2$, or $M'$ is $y'$-special and $\ell-2=2$ (so $m\geq 5$).

For the first case, since $m\ge 4$ and $M'$ is $y$-special,  
$y$ has two neighbors $z, w\in V(M')\setminus \{y,y'\}$ where $z'w' \in E(G)$. 
See the first figure of Figure~\ref{fig:path}.
Since $xx'$ and $zz'$ are part of a strong clique of $G$, there must be an edge between $\{x,x'\}$ and $\{z,z'\}$.
In each case, there is an $(x, M)$-path of length $3=\ell$, which is $xzz'w'$, $xz'zy$, $xx'zy$, or $xx'z'w'$. 

\begin{figure}[h!]
    \centering
    \includegraphics[height=2.5cm,page=9]{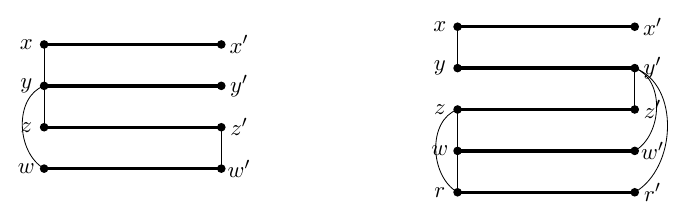}
    \caption{Illustrations for Lemma~\ref{lem:path1}}  \label{fig:path}
\end{figure}

For the second case, 
since $m\ge 5$ and $M'$ is $y'$-special, 
$y'$ has three neighbors $z', w', r'\in V(M')\setminus \{y,y'\}$ where $z$, $w$, and $r$ are pairwise adjacent to each other. 
See the second figure of Figure~\ref{fig:path}.
Since $xx'$ and $zz'$ are part of a strong clique of $G$, there must be an edge between $\{x,x'\}$ and $\{z,z'\}$.
In each case, there is an $(x, M)$-path of length $\ell=4$, which is $xzww'y'$, $xz'zwr$, $xx'zwr$, or $xx'z'zw$. 

Thus, there is an $(x,M)$-path of length $\ell$ for every $\ell\in\{1,\ldots, m-1\}$, and this completes the proof.
\end{proof}

\begin{lemma}\label{lem:matching}
For a graph $G$, let $M$ be a matching of $G$ that is also a strong clique of $G$.
If $|M|=2m\geq6$, then $G$ contains a $C_{2m}$.
\end{lemma}

\begin{proof}
For every $x\in V(M)$, 
let $W(x)$ be a maximum subset of $V(M)\cap N_G(x)$ such that $W(x)$ contains at most one endpoint of every edge in $M\setminus\{xx'\}$.
Also define $W'(x)=\{w'\mid w\in W(x)\}$.

\begin{claim}\label{claim:indep}
For every $x\in V(M)$, if $W(x)$ is not an independent set of $G$, then $G$ contains a $C_{2m}$.
\end{claim}
\begin{proof}
Let $x\in V(M)$ such that $W(x)$ is not an independent set of $G$.

First, suppose that $G[W(x)]$ contains a $P_3$, that is, $yz, zw \in E(G)$ for some $y,z,w \in W(x)$. 
Let $M' = M\setminus\{xx',yy',zz'\}$, so $|M'|=2m-3\ge 3$. 
Assume that there is a $(w, M')$-path $P$ of length $2m-4$ ending at $r\in V(M')\setminus\{w, w'\}$.
See the first figure of Figure~\ref{fig:matching1}.
Since $xx'$ and $rr'$ are part of a strong clique of $G$, there must be an edge between $\{x, x'\}$ and $\{r, r'\}$. 
In each case, by adding $xyzw$, $x'xzw$, $r'xzw$, or $r'x'xw$ to $P$, we obtain a $C_{2m}$. 

\begin{figure}[h!]
    \centering
    \includegraphics[height=3cm,page=10]{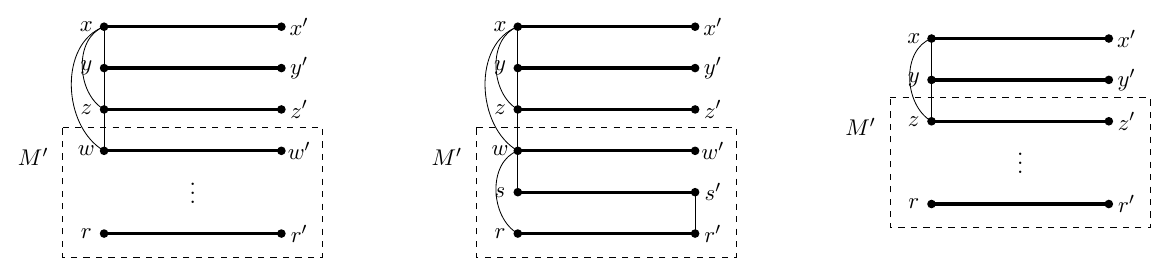}
    \caption{Illustrations for Claim~\ref{claim:indep}}  \label{fig:matching1}
\end{figure}

Hence, let us assume that there is no such path. 
By Observation~\ref{obs:special} and Lemma~\ref{lem:path1}, there are only two possible cases:
either $M'$ is $w$-special and $2m-4=2$, 
or $M'$ is $w'$-special and $2m-5=2$. 
Note that the second case is impossible.
For the first case,  let $M'=\{ww', ss',rr'\}$ where $wr,ws,s'r' \in E(G)$.
See the second figure of Figure~\ref{fig:matching1}.
Since $xx'$ and $rr'$ are part of a strong clique of $G$, there must be an edge between $\{x,x'\}$ and $\{r,r'\}$.
In each case, there is a $C_6$, which is $xrr's'swx$, $xr's'swzx$, $xx'rwzyx$, or $xx'r's'swx$. 
Therefore, $G$ contains a $C_{2m}$.

\smallskip

Now, suppose that $G[W(x)]$ has no $P_3$.
Since $W(x)$ is not an independent set of $G$, 
$yz \in E(G)$ for some $y,z \in W(x)$. Let $M'=M\setminus\{xx',yy'\}$, so  $|M'|=2m-2 \ge 4$.
Assume that there is a $(z,M')$-path $P$ of length $2m-4$ ending at $r \in V(M')\setminus \{z,z'\}$.
See the third figure of Figure~\ref{fig:matching1}.
Since $xx'$ and $rr'$ are part of a strong clique of $G$, there must be an edge between $\{x, x'\}$ and $\{r, r'\}$. 
If $rx\notin E(G)$, then we obtain a $C_{2m}$ by adding $x'xyz$, $r'xyz$, or $r'x'xz$ to $P$, depending on the case. 
So we may assume that $rx \in E(G)$. 
By symmetry, we may assume that $ry\in E(G)$, which is a contradiction to our assumption since $ryz$ is a $P_3$ in $G[W(x)]$.

Hence, let us assume that there is no such path. By Observation~\ref{obs:special} and Lemma~\ref{lem:path1}, there are two possible cases: either
$M'$ is $z$-special and $2m-4=2$, or $M'$ is $z'$-special and $2m-5=2$. 
Note that the second case is impossible.
For the first case,  
 $M'$ is a special matching of size $4$, so $G[V(M')]$ contains a $C_6$. 
See Figure~\ref{fig:sp_matching}.
Therefore, $G$ contains a $C_{2m}$.
\end{proof}

\begin{claim}\label{claim:indep2}
For every $x\in V(M)$, if $G[W'(x)]$ contains a $P_3$, then 
$G$ contains a $C_{2m}$.
\end{claim}
\begin{proof}
Suppose that for some  $x\in V(M)$, $G[W'(x)]$ contains a $P_3$, that is, $y'z', z'w' \in E(G)$ for some $y',z',w' \in W'(x)$. 
See the first figure of Figure~\ref{fig:matching2}.
If $m=3$, then $xyy'z'w'wx$ is a $C_{2m}$, so assume that $m\geq 4$. 
Note that $\{y,z,w\}$ is an independent set of $G$ by Claim~\ref{claim:indep}.
We may further assume that $y'w' \notin E(G)$ since otherwise $W(y')$ is not an independent set of $G$, which is a contradiction to Claim~\ref{claim:indep}. 
So, either $yw' \in E(G)$ or $wy' \in E(G)$. 
Without loss of generality, assume that $yw' \in E(G)$.

\begin{figure}[h!]
    \centering
    \includegraphics[height=3cm,page=11]{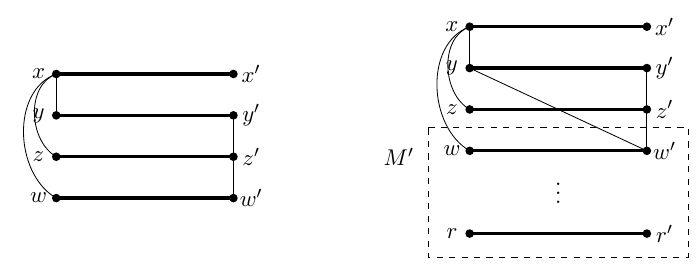}
    \caption{Illustrations for Claim~\ref{claim:indep2}}  \label{fig:matching2}
\end{figure}

Let $M'=M\setminus\{xx',yy',zz'\}$, so $|M'|=2m-3 \ge 5$. 
Assume that there exists a $(w',M')$-path $P$ of length $2m-5$ ending at $r\in V(M')\setminus\{w, w'\}$. 
See the second figure of Figure~\ref{fig:matching2}.
Since $xx'$ and $rr'$ are part of a strong clique of $G$, there must be an edge between $\{x,x'\}$ and $\{r,r'\}$.
In each case, by adding $xyy'z'w'$, $x'xzz'w'$, $r'xzz'w'$, or $r'x'xyw'$ to $P$, we obtain a $C_{2m}$.  

Hence, let us assume that there is no such path. 
By Observation~\ref{obs:special} and Lemma~\ref{lem:path1}, since $2m-5 \ge 3$, it must be that $M'$ is $w$-special and $2m-6=2$. 
Then, since 
$M'$ is a special matching of size $5$, $W(s)$ is not an independent set of $G$ for some vertex $s \in V(M')$.
See Figure~\ref{fig:sp_matching}.
Thus, $G$ contains a $C_{8}$ by Claim~\ref{claim:indep}.
Therefore, $G$ contains a $C_{2m}$.
\end{proof}

Let $x \in V(M)$ such that $|W({x})|$ is maximum.  
Note that $|W'(x)|=|W(x)|\ge m \ge 3$. 
From Claims~\ref{claim:indep}~and~\ref{claim:indep2}, we assume that 
$W(x)$ is an independent set of $G$ and $W'(x)$ has no $P_3$.

\begin{claim}\label{claim:indep3}
If $W'(x)$ is not an independent set of $G$, then $G$ contains a $C_{2m}$.
\end{claim}
\begin{proof} 
Suppose that  $W'(x)$ is not an independent set of $G$, that is, $y'z' \in E(G)$ for some $y',z' \in W'(x)$.
Let $w' \in W'(x)\setminus \{y',z'\}$.
See the first figure of Figure~\ref{fig:matching3}.
Recall that $\{y,z,w\}$ is an independent set of $G$, and $w'y', w'z' \notin E(G)$ since $W'(x)$ has no $P_3$.
If $w$ is adjacent to both $y'$ and $z'$, then $W(z')$ is not an independent set of $G$, which is a contradiction to Claim~\ref{claim:indep}. 
Hence, we may assume that either $wy' \notin E(G)$ or $wz'\notin E(G)$. 
Without loss of generality, assume that $wy' \notin E(G)$, which implies that $yw' \in E(G)$. 
Let $M'=M\setminus\{yy',zz',ww'\}$, so $|M'|=2m-3\geq 5$.

Suppose that $wz' \in E(G)$. 
If $m=3$, then $xyw'wz'z$ is a $C_{2m}$, so assume that $m\ge 4$.
See the first figure of Figure~\ref{fig:matching3}.
Assume that there is an $(x, M')$-path $P$ of length $2m-5$ ending at $r\in V(M')\setminus\{x, x'\}$.
See the second figure of Figure~\ref{fig:matching3}.
Since $rr'$ and $zz'$ are part of a strong clique of $G$, there must be an edge between $\{r,r'\}$ and $\{z,z'\}$.
In each case, by adding $zz'y'yx$, $z'ww'yx$, $r'zz'wx$, or $r'z'y'yx$ to $P$, we obtain a $C_{2m}$. 

\begin{figure}[h!]
    \centering
    \includegraphics[height=3cm,page=12]{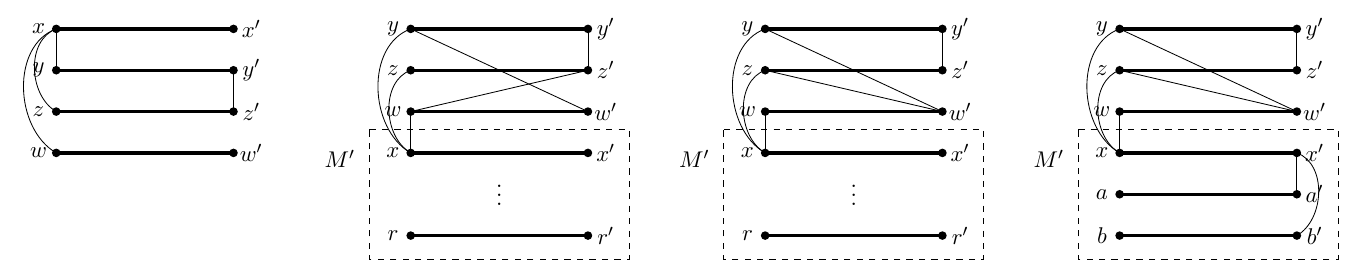}
    \caption{Illustrations for Claim~\ref{claim:indep3}}  \label{fig:matching3}
\end{figure}

Hence, let us assume that there is no such path.
By Observation~\ref{obs:special} and Lemma~\ref{lem:path1}, since $2m-5 \ge3$, 
it must be that $M'$ is $x'$-special and $2m-6=2$. 
Then, since $M'$ is a special matching of size $5$, $W(s)$ is not an independent set of $G$ for some vertex $s \in V(M')$.
See Figure~\ref{fig:sp_matching}.
Thus, $G$ contains a $C_{8}$ by Claim~\ref{claim:indep}.
Therefore, $G$ contains a $C_{2m}$.

Now, suppose that $wz'\notin E(G)$ so that $w'z\in E(G)$. 
Assume that there is an $(x, M')$-path $P$ of length $2m-5$ ending at $r\in V(M')\setminus\{x, x'\}$.
See the third figure of Figure~\ref{fig:matching3}.
Since $rr'$ and $zz'$ are part of a strong clique of $G$, there must be an edge between $\{r,r'\}$ and $\{z,z'\}$.
In each case, by adding $zz'y'yx$, $z'zw'wx$, $r'zw'wx$, or $r'z'y'yx$ to $P$, we obtain a $C_{2m}$.

Hence, let us assume that there is no such path. 
By Observation~\ref{obs:special} and Lemma~\ref{lem:path1}, there are three possible cases:
(i) $x$ has no neighbors in $V(M') \setminus \{x,x'\}$ and $2m-5=1$, (ii) $M'$ is $x$-special and $2m-5=2$, and (iii) $M'$ is $x'$-special and $2m-6=2$. 
Note that the second case is impossible.
For the first case,
let $\{a',b'\}=W(x')\setminus \{y,y',z,z',w,w'\}$. 
See the fourth figure of Figure~\ref{fig:matching3}.
Since $aa'$ and $yy'$ are part of a strong clique of $G$, there must be an edge between $\{a,a'\}$ and $\{y,y'\}$. 
If $ya,yb \in E(G)$, then $yaa'x'b'by$ is a $C_6$. 
Otherwise, we may assume without loss of generality that $ya \notin E(G)$, so one of the following is a $C_6$: $xx'a'yw'wx$, $xx'a'y'z'zx$, $xx'a'ay'yx$.
For the third case, since $G[W'(x')]$ contains a $P_3$, $G$ contains a $C_8$ by Claim~\ref{claim:indep2}. 
Therefore, $G$ contains a $C_{2m}$.
\end{proof}

For simplicity, let $W(x)=\{u_1,\ldots,u_n\}$. 
Note that $n\ge m \ge 3$.  
By Claim~\ref{claim:indep3}, we may assume that $W'(x)$ is an independent set of $G$.
Hence, $G[W(x)\cup W'(x)]$ is bipartite.
By Lemma~\ref{lem:matching:path},
we may assume, by relabeling indices if necessary, 
that $u_1u_1'u_2u_2'\ldots u_{m-1}u_{m-1}'u_m$ is a path in $G$.
Then, by adding two edges $u_mx$ and $xu_1$ to the path, we obtain a $C_{2m}$.
This proves Lemma~\ref{lem:matching}. 
\end{proof}

Now, we prove Theorem~\ref{thm:general}.

\begin{proof}[Proof of Theorem~\ref{thm:general}]
Let $G$ be a $C_{2k}$-free graph with $\Delta(G)\geq 1$, and $H$ be the subgraph of $G$ induced by a maximum strong clique of $G$.
Let $M$ be a maximum matching of $H$ and let $|M|=m$. 
If $m\geq 2k$, then $G$ contains a $C_{2k}$ by Lemma~\ref{lem:matching}. 
Hence, $m\le 2k-1$.  
If $\Delta(H)=1$, then $|E(H)|= |M|\le (2k-1)\Delta(G)$, so in what follows suppose that $\Delta(H)\ge 2$.

Define the following sets:
$Z=V(G)\setminus V(M)$, 
$X=\{ v\in V(M)\mid N_H(v)\cap Z \neq \emptyset\}$, and  $Y=\{ v\in V(M) \mid|N_H(v)\cap Z|\ge 2 \}$.
Also, let $D=\max\{ \deg_H(v)\mid v\in V(M)\setminus Y\}$, which implies $D\le \min\{ 2m,\Delta(H)\}$ since every vertex in $V(M)\setminus Y$ has at most one neighbor of $H$ in $Z$. 
Now,
\[\sum_{v\in V(H)}\deg_H(v)
\le |Y|\Delta(H)+(2m-|Y|)D + |E_H(V(M),Z)|.\]

By the maximality of $M$, if $xx'\in M$, then $x$ and $x'$ cannot have distinct neighbors of $H$ in $Z$. 
Thus every edge in $M$ has at most one endpoint in $Y$, so $|Y|\le m$.
Therefore, $|E_H(V(M),Z)|\le (\Delta(H)-1)|Y|+(2m-2|Y|)$.
Thus 
$2|E(H)|\le  2|Y|\Delta(H) +(2mD-|Y|D+2m-3|Y|)$, and so 
\begin{eqnarray}\label{eq:last}
&& |E(H)|\le  |Y|(\Delta(H)-D/2-1) +m(D+1). \end{eqnarray}
Note that $\Delta(H)-D/2-1\ge 0$ since $\Delta(H)\ge \frac{\Delta(H)+2}{2}\ge 1+\frac{D}{2}$.
From \eqref{eq:last} and the fact that $|Y|\le m$, $D\le 2m$, and $m\leq 2k-1$, we conclude the following:
\begin{eqnarray*}
|E(H)|  
 &\leq&  m(\Delta(H)-D/2-1)+m(D+1) \\ 
 &=&  m\Delta(H)+mD/2\\
 &\leq & (2k-1)\Delta(H) + (2k-1)^2\\
 &\leq & (2k-1)\Delta(G) + (2k-1)^2.
\end{eqnarray*}
\end{proof}
\begin{remark}
We can actually show $\SC(G)\le \frac{(2k-1)(4\Delta(G)+1)}{3}$;  this is a better bound than the upper bound obtained from Theorem~\ref{thm:general} if $\Delta(G)\le 12k-8$.
We provide an outline of the proof, and use the notation in the proof of Theorem~\ref{thm:general}. 

We may assume that $\Delta(H)\ge 2$.
By the maximality of $M$, there is no edge of $H$ between two vertices in $\{ y'\mid y\in Y\}$.
Thus, $V(M)\setminus\{ y'\mid y\in Y\}$ is a vertex cover of $H$.
If $|Y|\ge \frac{2m}{3}$, then  $|E(H)|\le (2m-|Y|)\Delta(H)\le \frac{4(2k-1)}{3}\Delta(H)$.
If $|Y|<\frac{2m}{3}$, then, from \eqref{eq:last}, $|E(H)| \le \frac{2m}{3}(\Delta(H)-1-\frac{D}{2})+m(D+1)$,
since $\Delta(H)-1-\frac{D}{2}\ge 0$.
Hence, $|E(H)|\le \frac{2m}{3}\Delta(H) +\frac{2mD}{3}+\frac{m}{3} \le \frac{4(2k-1)\Delta(H)+(2k-1)}{3}$.
\end{remark}

\section*{Acknowledgements}
\small{Ilkyoo Choi was supported by the Basic Science Research Program through the National Research Foundation of Korea funded by the Ministry of Education (No. NRF-2018R1D1A1B07043049), and also by the Hankuk University of Foreign Studies Research Fund.
Ringi Kim was supported by the National Research Foundation of Korea grant funded by the Korea government (No. NRF-2018R1C1B6003786), and also by Basic Science Research Program through the National Research Foundation of Korea funded by the Ministry of Education (No. NRF-2019R1A6A1A10073887).
Boram Park was supported by the National Research Foundation of Korea grant funded by the Korea government (No. NRF-2018R1C1B6003577).}

\bibliography{ref}{}
\bibliographystyle{plain}
 \end{document}